\documentclass{elsarticle}
\usepackage{amsmath}
\usepackage{amsthm}
\usepackage{amssymb}
\usepackage{mathrsfs}
\usepackage{hyperref}
\usepackage[curve,arrow,matrix]{xy}

\theoremstyle{plain}
\newtheorem{thm}{Theorem}[section]
\newtheorem{prop}[thm]{Proposition}
\newtheorem{lemma}[thm]{Lemma}

\newtheorem{cor}[thm]{Corollary}

\newcommand{\deqno}{\refstepcounter{thm}(\thethm)}

\theoremstyle{definition}
\newtheorem{defn}[thm]{Definition}
\newtheorem{disc}[thm]{Discussion}

\newtheorem{exam}[thm]{Example}

\theoremstyle{remark}
\newtheorem{remk}[thm]{Remark}

\newcounter{item}
\newenvironment{rlist}[1][1]{\begin{list}
  {\textup{(\roman{item})}}{\usecounter{item} \setcounter{item}{#1}\addtocounter{item}{-1}
  \setlength{\itemsep}{0ex}
  \setlength{\topsep}{0ex} \setlength{\parsep}{0ex} \setlength{\labelwidth}{15mm}
   \setlength{\leftmargin}{10mm} } }{\end{list}}
\newenvironment{alist}[1][1]{\begin{list}
  {\textup{(\alph{item})}}{\usecounter{item} \setcounter{item}{#1}\addtocounter{item}{-1}
  \setlength{\itemsep}{0ex}
  \setlength{\topsep}{0ex} \setlength{\parsep}{0ex} \setlength{\labelwidth}{15mm}
  \setlength{\leftmargin}{10mm} } }{\end{list}}
\newenvironment{nlist}[1][1]{\begin{list}
  {\textup{(\arabic{item})}}{\usecounter{item} \setcounter{item}{#1}\addtocounter{item}{-1}
  \setlength{\itemsep}{0ex}
  \setlength{\topsep}{0ex} \setlength{\parsep}{0ex} \setlength{\labelwidth}{15mm}
  \setlength{\leftmargin}{10mm} } }{\end{list}}

\newcommand{\tensor}{\otimes}
\newcommand{\Hom}{\textup{Hom}}
\newcommand{\Ext}{\textup{Ext}}

\newcommand{\im}{\textup{Im}}
\newcommand{\id}{\textup{id}}

\newcommand{\ra}{\rightarrow}

\newcommand{\isom}{\cong}

\newcommand{\wt}{\widetilde}

\begin{document}

\title{Axiomatic Closure Operations, Phantom Extensions, and Solidity}

\author[gdd]{Geoffrey D. Dietz}
\ead{gdietz@member.ams.org}
\address[gdd]{Department of Mathematics,
Gannon University, Erie, PA 16541}


\date{\today}

\begin{abstract}
In this article, we generalize a previously defined set of axioms for a closure operation that induces balanced big Cohen-Macaulay modules. While the original axioms were only defined in terms of finitely generated modules, these new ones will apply to all modules over a local domain.  The new axioms will lead to a notion of phantom extensions for general modules, and we will prove that all modules that are phantom extensions can be modified into balanced big Cohen-Macaulay modules and are also solid modules. As a corollary, if $R$ has characteristic $p>0$ and is $F$-finite, then all solid algebras are phantom extensions. If $R$ also has a big test element (e.g., if $R$ is complete), then solid algebras can be modified into balanced big Cohen-Macaulay modules. (Hochster and Huneke have previously demonstrated that there exist solid algebras that cannot be modified into balanced big Cohen-Macaulay algebras.)  We also point out that tight closure over local domains in characteristic $p$ generally satisfies the new axioms and that the existence of a big Cohen-Macaulay module induces a closure operation satisfying the new axioms.
\end{abstract}

\maketitle

Let $R$ be a local domain. In \cite[(1.1)]{D10}, we presented seven axioms for a closure operation that were necessary and sufficient in order to prove the existence of balanced big Cohen-Macaulay modules over $R$. \cite[(3.16)]{D10} (Recall that a balanced big Cohen-Macaulay module $B$ over $R$ is a module for which every partial system of parameters for $R$ is a regular sequence on $B$.) The work is independent of characteristic considerations for $R$. Tight closure in equal characteristic satisfies these axioms. The existence of a big Cohen-Macaulay module or algebra over $R$ implies the existence of such a closure operation in any characteristic in mixed characteristic. Whether or not there exist big Cohen-Macaulay modules or algebras in mixed characteristic was in general a major open problem for many years, but the issue has been recently settled via the use of perfectoids. See \cite{A16}, \cite{B16}, \cite{HM17}. For rings of Krull dimension at most 3, see \cite{He02,Ho02} for older results in the literature.

In the present article, we will generalize the seven axioms to include closure within general $R$-modules, not just finitely generated modules. The axioms of \cite[(1.1)]{D10} will be termed \textit{finite axioms} while these new axioms will be called the \textit{big axioms}. We will show that tight closure in characteristic $p$ satisfies the big axioms (Proposition~\ref{tcbigaxioms}) and that the existence of a balanced big Cohen-Macaulay $R$-module induces a closure operation on modules that satisfies the big axioms (Proposition~\ref{bigaxiomsandbigCM}). These two results are analogous to \cite[(5.4), (4.2)]{D10} that were proven for the finite axioms. Given a closure operation satisfying the finite axioms, we prove (Proposition~\ref{finext}) that this operation can be extended to a closure for all $R$-modules that satisfies the big axioms. The technique is based on the idea of finitistic tight closure in \cite[Definition 8.19]{HH90}. 

In the second section, we will use the big axioms to define when a general $R$-module $M$ is a phantom extension of $R$ with respect to the closure operation. Although this definition appears to depend upon choices of a free presentation and lifts of certain maps, we prove that it is in fact independent of all choices. In the following section, we will prove that all phantom extensions can be modified into balanced big Cohen-Macaulay $R$-modules (Theorem~\ref{phmodbigCM}), which is analogous to \cite[(3.16)]{D10} showing that the finite notion of phantom extensions can be modified into big Cohen-Macaulay modules. In the fourth section, we use the earlier results to prove that if $\alpha:R\to M$ is phantom over a 
local domain, then $M$ is a solid module. We also prove that in prime characteristic $p$ over an $F$-finite ring $R$, solid algebras are phantom extensions (Theorem~\ref{solidalgph}). The resulting Corollary~\ref{solidalgbigCM} is that solid algebras can be modified into big Cohen-Macaulay modules, which complements the example of Hochster and Huneke \cite[Example 3.14]{HH09} of a solid algebra in characteristic $p$ that cannot be modified into a big Cohen-Macaulay algebra. On the other hand, we show that there are solid modules that cannot be phantom extensions with respect to any closure operation satisfying the axioms.

\section{Big and Finite Closure Axioms}

We will alter the notation of \cite[(1.1)]{D10} (e.g., $N_M^\natural$) and adopt some of the notation and naming conventions used in \cite{E12} to emphasize the generic, axiomatic nature of our closure operations and to avoid confusion with other closure operations that may use the $\natural$ symbol as well. 

\begin{defn} \label{bigaxioms}
Let $(R, \mathfrak{m})$ be a local domain. All modules named below are arbitrary and not necessarily finitely generated. 
The following is a list of axioms that may be satisfied by an operation $cl$ on submodules $N$ of $M$. 
\begin{nlist}
  
  \item (extension) $N_M^{cl}$ is a submodule of $M$ containing $N$. 
  
  \item (idempotence) $(N_M^{cl})_M^{cl} = N_M^{cl}$. 

  \item (order-preservation) If $N\subseteq M\subseteq W$, 
  then $N_W^{cl} \subseteq M_W^{cl}$. 
  
  \item (functoriality) Let $f:M\ra W$ be a homomorphism. 
  Then $f(N_M^{cl})\subseteq f(N)_W^{cl}$.
  
  \item (semi-residuality) If $N_{M}^{cl} =N$, then $0_{M/N}^{cl} = 0$.

  \item The ideal $\mathfrak{m}$ 
  is closed; \emph{i.e.},
  $\mathfrak{m}^{cl}_R = \mathfrak{m}$.
 
  \item (generalized colon-capturing) Let $x_1,\ldots, x_{k+1}$ be a partial system of parameters for
  $R$, and let $J=(x_1,\ldots,x_k)$. Suppose that there exists a surjective
  homomorphism $f:M\to R/J$ and $v\in M$ such that $f(v) = x_{k+1} + J$. Then
  $$
  (Rv)^{cl}_M \cap \textrm{ker}\, f \subseteq (Jv)^{cl}_M.
  $$

\end{nlist} 
We will refer to the list above as the \textit{big axioms}. If the closure operation only satisfies the axioms for finitely generated modules, then we say that the operation meets the \textit{finite axioms}. 
\end{defn}

\begin{remk}
Note that we removed the condition $0^{cl}_R = 0$ from the sixth axiom. A close reading of \cite{D10} shows that the only citation of that axiom was in Lemma~3.12, but changing the equality at the end of the displayed formula to an inclusion $\supseteq$ still yields the result but does so using axiom~(1) instead of $0^{cl}_R = 0$.
\end{remk}

It is worth pointing out that the results of \cite[Lemmas 1.2 and 1.3, Proposition 1.4]{D10} concerning quotients, direct sums, intersections, sums, and colon-capturing are still valid for an operation that satisfies the big axioms. The proofs are identical as they make no reference to finite generation of modules. We restate a few for ease of reference and add to that list a result related to isomorphisms.  Additionally, Mel Hochster and Rebecca R.G.\ pointed out that the condition $0^{cl}_R = 0$ follows from the other axioms. 

\begin{lemma}\label{axiomprop}
Let $cl$ be an operation satisfying the big (or finite) axioms. All $R$-modules are arbitrary (or finite). 
\begin{alist}
\item \cite[Lemma 1.2(a)]{D10} Let $N'\subseteq N \subseteq M$ be $R$-modules. Then $u\in N^{cl}_M$ if and only if $u+N'\in (N/N')^{cl}_{M/N'}$.
\item \cite[Lemma 1.2(b)]{D10} If $N_i\subseteq M_i$, then $(N_1\oplus N_2)^{cl}_{M_1\oplus M_2} = (N_1)^{cl}_{M_1} \oplus (N_2)^{cl}_{M_2}$. In the case of the big axioms, this result also applies to infinite direct sums.
\item \cite[Lemma 1.3]{D10} Let $x_1,\ldots, x_{k+1}$ be a partial system of parameters for $R$, and let $J=(x_1,\ldots,x_k)$. Suppose that there exists a homomorphism $f:M\to R/J$ and $v\in M$ such that $f(v) = x_{k+1} + J$. Then $(Rv)^{cl}_M \cap \textrm{ker}\, f \subseteq (Jv)^{cl}_M$.
\item If $f:M\to M'$ is an isomorphism of $R$-modules with $N\subseteq M$ and $N' = f(N)$, then $u\in N_M^{cl}$ if and only if $f(u) \in (N')_{M'}^{cl}$. 
\item $0^{cl}_R = 0$.
\item Let $u\in N_M^{cl}$. Then there exists $r\neq 0$ in $R$ such that $ru \in N$. 
\end{alist}
\end{lemma}
\begin{proof}
For (a), (b), and (c), use the existing proofs in \cite{D10} as finiteness of modules is irrelevant to the proofs. For (d), apply axiom (4) to $f$ and to $f^{-1}$ to get the two containments.

\footnote{The proof of this result comes from Mel Hochster and Rebecca R.G.}For (e), first note that if $R$ is a local domain of Krull dimension 0, then $0 = \mathfrak{m}$ and so is closed by axiom (6). If $R$ has Krull dimension greater than 0, then suppose that $x\in 0^{cl}_R\setminus 0$. We can consider $x=x_1$ a partial s.o.p., and let $J=(0)$ be the ideal generated by the ``other'' parameters in the partial system. Let $M = R\oplus R$ and define $f:M\to R = R/J$ be $f(r,s) = rx+s$ so that $v = (1,0)\mapsto x_1$. Then $f(-1,x) = 0$, and $(-1,x) \in (Rv)^{cl}_M$ as $(Rv)_M^{cl} = (R\oplus 0)^{cl}_{R\oplus R} = R^{cl}_R \oplus 0^{cl}_R = R\oplus 0^{cl}_R$, which contains $(-1,0)$ and $(0,x)$. Hence, axioms~(6) and (7) imply that $(-1,x)\in (Jv)_M^{cl} = 0_M^{cl} = 0^{cl}_R \oplus 0^{cl}_R \subseteq \mathfrak{m}\oplus R$, which is impossible.

\footnote{The proof of this result arose from a conversation with Rebecca R.G.}For (f), let $M'$ be any module. Let $T$ be the submodule of torsion elements. Then $M'/T$ is torsion-free. By \cite[(3.1g)]{RG15b}, $0^{cl}_{M'/T} = 0$, which implies $T^{cl}_{M'} = T$ by part (a). As $0\subseteq T$, $0^{cl}_{M'} \subseteq T$ by the order-preservation axiom. Hence, if $u\in 0^{cl}_{M'}$, then $u\in T$ and so there exists $r\neq 0$ such that $ru = 0$. Now let $N\subseteq M$ be any modules and consider $M' = M/N$ and part (a). If $u\in N^{cl}_M$, then there exists an $r\neq 0$ such that $ru\in N$.
\end{proof}

The axioms and theorems developed in \cite{D10} are results of the finite axioms. In this paper we will develop analogues and new results for the big axioms. 
Although the big axioms may appear more powerful (or more restrictive) at first compared to the finite axioms, they are equivalent to each other in the following sense.

\begin{prop} \label{bigaxiomsandbigCM}
Let $(R, \mathfrak{m})$ be a local domain. The following are equivalent:
\begin{rlist}
\item There exists a closure operation $c$ that satisfies the big axioms (1)--(7).
\item There exists a closure operation $d$ that satisfies the finite axioms (1)--(7).
\item There exists a balanced big Cohen-Macaulay module $B$ over $R$.
\end{rlist}
\end{prop}
\begin{proof}
(i) implies (ii) is obvious with $d=c$. (ii) implies (iii) is the main result (Theorem 3.16) of \cite{D10}. To see (iii) implies (i), we just note that the proof of Theorem 4.2 in \cite{D10} suffices as finite generation of modules is irrelevant to the arguments, even though the theorem is stated to prove the existence of an operation meeting the finite axioms. 
\end{proof}

Although the proposition above implies that a closure operation $d$ meeting the finite axioms implies the existence of an operation $c$ meeting the big axioms, there is nothing guaranteeing that $c$ and $d$ agree for finitely generated modules. Indeed, the closure $d$ induces a big Cohen-Macaulay module $B$, which in turn induces a closure operation $c$ using extension and contraction from $B$, but those need not be the same. On the other hand, an operation satisfying the finite axioms can be extended to one meeting the big axioms using the idea of finitistic tight closure given in \cite[Definition 8.19]{HH90}. 

\begin{prop} \label{finext}
Let $(R, \mathfrak{m})$ be a local domain. Suppose that $cl$ is an operation on modules that satisfies the finite axioms. Then $cl$ can be extended to all modules so that the extended operation satisfies the big axioms.
\end{prop}
\begin{proof}
Given a notion of $N_M^{cl}$ for finitely generated modules $N\subseteq M$ that satisfies the finite axioms, we extend $cl$ to all modules $N\subseteq M$ via
$u\in N_M^{cl}$ if and only if there exists a finitely generated submodule $M_0 \subseteq M$ such that $u \in (N\cap M_0)_{M_0}^{cl}$. This definition leaves $N_M^{cl}$ unchanged for finitely generated modules $M$ due to the inclusion $\iota: M_0\to M$, the containments $N\cap M_0 \subseteq N \subseteq M$, and finite axioms (4) and (3). Thus, this definition is an extension of $cl$ to arbitrary modules. 

Let $N\subseteq M$ be arbitrary modules. We check that the big axioms hold. 
\begin{nlist}
\item If $u\in N$, then $u\in (N\cap Ru)_{Ru}^{cl}$ by finite axiom (1) so that $N\subseteq N_M^{cl}$. Let $r\in R$ and $u_1, u_2 \in N_M^{cl}$. Then there exist finitely generated submodules $M_1$, $M_2$ of $M$ such that $u_i \in (N\cap M_i)^{cl}_{M_i}$. Let $M_0 = M_1 + M_2$, a finitely generated submodule of $M$. Then finite axioms (3) and (4) imply that $u_i \in (N\cap M_0)^{cl}_{M_0}$. Thus, finite axiom (1) implies that $ru_1 + u_2 \in (N\cap M_0)^{cl}_{M_0}$ and so $ru_1 + u_2 \in N_M^{cl}$. 
\item Suppose that $u\in (N_M^{cl})_M^{cl}$. We claim that $u\in N_M^{cl}$. There exists a finitely generated submodule $M_0\subseteq M$ such that $u\in (N_M^{cl}\cap M_0)_{M_0}^{cl}$. Let $w_1,\ldots, w_k$ be generators of the finitely generated submodule $N_M^{cl}\cap M_0$ of $M_0$. Then for each $j$ there exists a finitely generated submodule $M_j$ of $M$ such that $w_j\in (N\cap M_j)_{M_j}^{cl}$ as each $w_j$ is in $N_M^{cl}$. Let $Q = M_0 + M_1 + \cdots + M_k$, a finitely generated submodule of $M$. By finite axioms (4) and (3), $w_j \in (N\cap M_j)_{M_j}^{cl} \subseteq (N\cap M_j)_Q^{cl} \subseteq (N\cap Q)_Q^{cl}$ for each $j$, and so $N_M^{cl}\cap M_0 \subseteq (N\cap Q)_Q^{cl}$. Thus, finite axioms (4), (3), and (2) imply that
$$
(N_M^{cl}\cap M_0)_{M_0}^{cl} \subseteq (N_M^{cl}\cap M_0)_{Q}^{cl}  \subseteq ((N\cap Q)_Q^{cl})_Q^{cl} = (N\cap Q)_Q^{cl}
$$
which means that $u\in (N\cap Q)_Q^{cl}$ and so $u\in N_M^{cl}$ as claimed.
\item follows from finite axioms (3) as $(N\cap W_0)^{cl}_{W_0} \subseteq (M\cap W_0)^{cl}_{W_0}$ for any finitely generated submodule $W_0$ of $W$.
\item follows from finite axioms (3) and (4) as $M_0$ finitely generated in $M$ implies that $f(M_0)$ is a finitely generated submodule of $W$ and $f(N\cap M_0) \subseteq f(N) \cap f(M_0)$
\item Suppose that $N_M^{cl} = N$ and $\overline{u} \in 0_{M/N}^{cl}$ where $u\in M$. Then there exists a finitely generated submodule $M_0$ of $M$ such that $\overline{u} \in 0^{cl}_{\frac{M_0+N}{N}}$. Thus big axiom (4) and Lemma~\ref{axiomprop}(d) imply that $\overline{u} \in 0^{cl}_{\frac{M_0}{M_0\cap N}}$. Then Lemma~\ref{axiomprop}(a) applied to the finite module $\frac{M_0}{M_0\cap N}$ implies that $u\in (N\cap M_0)^{cl}_{M_0}$ and so $u\in N_M^{cl} = N$.
\item obvious
\item Let $x_1,\ldots x_{k+1}$ be a partial system of parameters, $J = (x_1,\ldots, x_k)$, $M$ an arbitrary $R$-module, $f:M\to R/J$ with $v\in M$ such that $f(v) = \overline{x_{k+1}}$. Suppose that $u\in (Rv)_M^{cl}$ and $f(u) = 0$. Then there exists a finite module $M_0\subseteq M$ such that $u\in (Rv\cap M_0)_{M_0}^{cl}$. Using finite axioms (3) and (4), we can assume without loss of generality that $v\in M_0$ and so $u\in (Rv)_{M_0}^{cl}$. Restricting $f$ to $M_0$ and applying Lemma~\ref{axiomprop}(c) shows that $u \in (Jv)_{M_0}^{cl}$ and so $u\in (Jv)_M^{cl}$. 
\end{nlist}
\end{proof}
 
As tight closure in characteristic $p$ over a complete local domain defined for finitely-generated modules satisfies the finite axioms \cite[Example 5.4]{D10}, finitistic tight closure satisfies the big axioms by the result above. We also note that tight closure in characteristic $p$ defined for arbitrary modules satisfies the big axioms, but the proof requires the use of a big test element in order to prove idempotence.

Let $(R,\mathfrak{m})$ be a local domain of prime characteristic $p$. Let  $\mathcal{F}$ denote the Frobenius functor on $R$-modules induced by the Frobenius endomorphism for $R$. Let $\mathcal{F}^e$ represent the iterated Frobenius functor, and let $N_M^{[q]} = \im(F^e(N) \to F^e(M))$. Then $u\in N^*_M$ if and only if there exists $c\ne 0$ such that $c\tensor u = cu^q \in N_M^{[q]}$ for all $e\geq 0$. See \cite[Section 8]{HH90} or \cite{Ho07} for further details on tight closure of modules. We say that $R$ possesses a big test element $c$ if the same $c$ can be used in all tight closure calculations for all pairs of modules $N\subseteq M$. Big test elements are known to exist in a large number of cases, e.g., if $R$ is also complete. See \cite[p.\ 150]{Ho07}. Note that we are not using finitistic tight closure and that whether these two concepts are the same or not is still an open problem.

\begin{prop} \label{tcbigaxioms}
Let $(R,\mathfrak{m})$ be a local domain of prime characteristic $p$ that possesses a big test element, e.g., $R$ is also complete. Then tight closure defined as above for all $R$-modules obeys the big axioms.
\end{prop}
\begin{proof}
Let $N\subseteq M$ be arbitrary $R$-modules.
\begin{nlist}
\item Proposition 8.5(a) of \cite{HH90}
\item Let $c$ be a big test element for $R$. If $u\in (N_M^*)^*_M$, then $cu^q \in (N_M^*)_M^{[q]}$. Thus, $cu^q = r_1u_1^q + \cdots r_nu_n^q$, where $r_i\in R$ and $u_i\in N^*_M$. Then $cu_i^q \in N^{[q]}_M$, and $c^2u^q \in N^{[q]}_M$ for all $q$, which implies that $u\in N^*_M$. (\textit{cf.} \cite[p164]{Ho07})
\item Proposition 8.5(b) of \cite{HH90}
\item Lemma 5.5 of \cite{D10} gives a valid proof, even though the statement unnecessarily presumes the modules are finitely generated.
\item Remark 8.4 in  \cite{HH90}
\item 
$\mathfrak{m}$ is a prime ideal, which is always tightly closed.
\item The proof of Proposition 5.6 of \cite{D10} still applies when $M$ is not necessarily finite. 
\end{nlist}
\end{proof}

\section{Phantom Extensions}

In this section we will generalize and refine the notion of phantom extensions introduced in \cite[Section 5]{HH94} with respect to tight closure for prime characteristic rings and generalized in \cite{D10} to any closure operation obeying the finite axioms independent of characteristic. We expand the notion to arbitrary $R$-modules and remove references to matrix representations of maps. In the end, the description below is equivalent to the other two in the cases of characteristic $p$ or finite modules. Throughout this section $(R,\mathfrak{m})$ is a local domain.

\begin{disc} \label{freepres}
Let $M$ be an arbitrary $R$-module along with an injection $\alpha: R\to M$ yielding the short exact sequence $0\to R\overset{\alpha}\to M\overset{\beta}\to Q\to 0$. Let $G\overset{\nu}\to F\overset{\mu}\to Q\to 0$ be any free presentation of $Q$. We will construct maps connecting the free presentation of $Q$ to the given short exact sequence. Although the methods are standard and well-known, we will take care in the construction so that we are fully aware of where choices are made in the process. 

Start by choose a lifting $\wt\mu: F\to M$ of $\mu$ such that $\mu = \beta\circ \wt\mu$. Consider the map $\wt\mu\circ \nu: G\to M$. Then $\beta\circ (\wt\mu\circ \nu) = \mu\circ \nu = 0$, and so $\wt\mu\circ \nu$ maps $G$ into the $\ker \beta = \im\ \alpha$. 
Hence $\wt\mu\circ \nu: G\to \im\ \alpha\subseteq M$. Since $\alpha$ is injective, we have a map $\wt\nu = \alpha^{-1}\circ \wt\mu\circ \nu: G\to R$. 
All told, we have a commutative diagram with exact rows:
$$
\xymatrix{ 0\ar[r] & R \ar[r]^\alpha & M \ar[r]^\beta & Q\ar[r] & 0 \\
& G\ar[r]^\nu \ar[u]_{\wt\nu} & F\ar[r]^\mu\ar[u]_{\wt\mu} & 
Q\ar[r]\ar[u]_\id & 0} \leqno \deqno \label{dgrm}
$$
Note that given the injection $\alpha: R\to M$, we made only two choices in constructing the diagram above: the choice of a free presentation and the choice of the lifting $\wt\mu$; the map $\wt\nu$ was uniquely determined by the other choices.

Let $(-)^\vee$ denote the dual operator $\Hom_R(-,R)$. Then $\nu^\vee: F^\vee \to G^\vee$ is the map sending $h\in F^\vee$ to $h\circ\nu \in G^\vee$, and $\wt\nu \in G^\vee$. 
\end{disc}

\begin{defn}
Let $cl$ be an operation that obeys the big axioms. Given an injection $\alpha:R\to M$ and a diagram as in (\ref{dgrm}), we say that $M$ is a \textit{phantom extension of $R$ via $\alpha$} if $\wt\nu \in (\im\ \nu^\vee)^{cl}_{G^\vee}$.
\end{defn}

\begin{remk}
This definition is equivalent in characteristic $p$ for finite modules to that given in \cite[Definition 5.2]{HH94} due to Theorem 5.13 of the same article. It is also equivalent to the definition given in \cite[Definition 2.2]{D10} with a different choice for notation. Essentially, to say that $\alpha$ is a phantom extension means that the cycle in $\Ext^1_R(Q,R)$ that corresponds to the SES built from $\alpha$ also lives in the closure of the boundaries (calculated within $G^\vee$). We have chosen, however, not to refer to membership in the module $\Ext^1_R(Q,R)$ but leave everything in terms of diagrams instead. 
\end{remk}

The definition is independent of both the choice of presentation and the lift $\wt\mu$.

\begin{lemma} \label{indep}
Let $cl$ be an operation obeying the big axioms, and let $\alpha:R\to M$ be an injection. If $\wt\nu \in (\im\ \nu^\vee)^{cl}_{G^\vee}$ for some choice of free presentation and some choice of $\wt\mu$, then it is true for all free presentations and choices of $\wt\mu$. 
\end{lemma}
\begin{proof}
First suppose that a choice of free presentation has been made, but two choices of lifts $\wt\mu$ and $\wt\mu'$ have been made, along with corresponding $\wt\nu$ and $\wt\nu'$.
$$
\xymatrix{ R \ar[r]^\alpha & M \ar[r]^\beta & Q\ar[r] & 0 \\
G\ar[r]^\nu \ar@<-.5ex>[u]_{\wt\nu'} \ar@<.5ex>[u]^{\wt\nu} & F\ar[r]^\mu \ar@<-.5ex>[u]_{\wt\mu'} \ar@<.5ex>[u]^{\wt\mu} & 
Q\ar[r]\ar[u]_\id & 0} 
$$
Since $\beta\circ \wt\mu = \mu = \beta\circ \wt\mu'$, we have $\wt\mu - \wt\mu': F\to \ker\beta = \im\ \alpha \subseteq M$ and so $\alpha^{-1}\circ (\wt\mu - \wt\mu'): F\to R$ is a well-defined map. Moreover, 
$$\nu^\vee(\alpha^{-1}\circ (\wt\mu - \wt\mu') )= \alpha^{-1}\circ (\wt\mu - \wt\mu')\circ \nu = \wt\nu - \wt\nu'$$ 
so that $\wt\nu - \wt\nu'\in \im\ \nu^\vee \subseteq (\im\ \nu^\vee)^{cl}_{G^\vee}$ by axiom (1). Therefore, $\wt\nu \in (\im\ \nu^\vee)^{cl}_{G^\vee}$ if and only if $\wt\nu' \in (\im\ \nu^\vee)^{cl}_{G^\vee}$.

Now suppose that $\wt\nu \in (\im\ \nu^\vee)^{cl}_{G^\vee}$ for a certain diagram (\ref{dgrm}) and that $G'\overset{\nu'}\to F'\overset{\mu'}\to Q\to 0$ is another free presentation of $Q$. Since $G'$ and $F'$ are free, we can build a commutative diagram
$$
\xymatrix{ R \ar[r]^\alpha & M \ar[r]^\beta & Q\ar[r] & 0 \\
G\ar[r]^\nu \ar[u]_{\wt\nu} & F\ar[r]^\mu\ar[u]_{\wt\mu} & Q\ar[r]\ar[u]_\id & 0\\
G'\ar[r]^{\nu'}  \ar[u]_{\phi} & F'\ar[r]^{\mu'} \ar[u]_{\psi} & Q\ar[r]\ar[u]_\id & 0
}
$$
Let $\wt\mu' = \wt\mu\circ \psi$ and let $\wt\nu' = \wt\nu\circ \phi$. Consider the dual map $\phi^\vee: G^\vee \to (G')^\vee$, where $\wt\nu' = \phi^\vee(\wt\nu)$ and $\phi^\vee(\im\ \nu^\vee) \subseteq \im\ (\nu')^\vee$ as 
$$
\phi^\vee\circ \nu^\vee(h) = h\circ\nu \circ \phi = h\circ \psi\circ \nu' = (\nu')^\vee(h\circ \psi)
$$ 
for all $h:F\to R$. Therefore, axioms (3) and (4) imply that $\wt\nu' \in (\im\ (\nu')^\vee)^{cl}_{(G')^\vee}$. Since this latter condition is true for this particular $\wt\mu'$ and $\wt\nu'$, then it is true for all choices of such maps using the free presentation $G'\overset{\nu'}\to F'\overset{\mu'}\to Q\to 0$ according to the first part of the proof. Therefore, the inclusion $\wt\nu \in (\im\ \nu^\vee)^{cl}_{G^\vee}$ (or non-inclusion) is independent of the choice of free presentation for $Q$ and choice of lifting $\wt\mu$.
\end{proof}

\begin{exam}
Notice that if $\alpha:R\to R$ is the identity map, then $Q=0$ and one can choose $G=F=0$ and $\wt\nu = 0$-map. It is thus trivially true that $\wt\nu \in (\im\ \nu^\vee)^{cl}_{G^\vee}$ and so $R$ is trivially a phantom extension of itself via the identity map.
\end{exam}

\begin{disc}
Given an injection $\alpha:R\to M$, expand diagram (\ref{dgrm}) to
$$
\xymatrix{ R \ar[r]^\alpha & M \ar[r]^\beta & Q\ar[r] & 0 \\
G\ar[r]^\nu \ar[u]_{\wt\nu} & F\ar[r]^\mu\ar[u]_{\wt\mu} & 
Q\ar[r]\ar[u]_\id & 0\\
G\ar[r]^{\nu_1} \ar[u]_{\id} & F\oplus R\ar[r]^{\mu_1} \ar[u]_{\pi} & 
M\ar[r]\ar[u]_\beta & 0
} \leqno \deqno \label{bigdgrm}
$$
where $\pi(f,r) = f$, $\mu_1(f,r) = \wt\mu(f) - \alpha(r)$, and $\nu_1(g) = (\nu(g), \wt\nu(g))$. 

The diagram commutes because $\pi\circ\nu_1(g) = \nu(g) = \nu\circ\id(g)$, and 
$$
\beta\circ\mu_1(f,r) = \beta\circ\wt\mu(f) - \beta\circ\alpha(r) = \mu(f) - 0 = \mu\circ\pi(f,r).
$$

The bottom row is also exact. The map $\mu_1$ is surjective as a given $m\in M$ leads to $\beta(m)= \mu(f) \in Q$, and $\beta(m - \wt\mu(f)) = \mu(f) - \mu(f) = 0$ leads to $m-\wt\mu(f) = \alpha(r)$ for an $r\in R$. Thus, $m = \wt\mu(f) +\alpha(r) = \mu_1(f,-r)$. 
We next show that we have exactness in the middle. First, $\mu_1\circ \nu_1(g) = \wt\mu(\nu(g)) - \alpha(\wt\nu(g)) = 0$. Second, if $\mu_1(f,r) = 0$, then $\wt\mu(f) = \alpha(r)$ and so $\mu(f) = \beta(\wt\mu(f)) = \beta(\alpha(r)) = 0$. Hence, $f = \nu(g)$ for some $g\in G$, and $\wt\nu(g) = \alpha^{-1}\circ \wt\mu\circ \nu(g) = \alpha^{-1}(\wt\mu(f)) = \alpha^{-1}(\alpha(r)) = r$. Thus, $\nu_1(g) = (\nu(g),\wt\nu(g)) = (f,r)$, and so the bottom row is exact.

We have proven the following.
\end{disc}

\begin{lemma}
Given an injection $\alpha:R\to M$, one can build a commutative diagram (\ref{bigdgrm}) with exact rows.
\end{lemma}

This diagram will allow us to understand $\alpha(1)$ within $M$ better.

\begin{lemma}
Let $cl$ be an operation satisfying the big axioms. Let $\alpha:R\to M$ be an injection with associated diagram (\ref{bigdgrm}). If $\alpha(1)\in \mathfrak{m}M$, then $\wt\nu$ is onto. The converse is true when $\im\ \nu \subseteq \mathfrak{m}F$, e.g., the free presentation is minimal. 
\end{lemma}
\begin{proof}
Suppose that $\alpha(1) \in \mathfrak{m}M$. Then $\alpha(1) \in \mathfrak{m}\im\ \wt\mu + \mathfrak{m}\im\ \alpha \subseteq \im\ \wt\mu + \mathfrak{m}\alpha(1)$, which implies that $\alpha(1) = \wt\mu(f_1) + x\alpha(1)$, where $x\in\mathfrak{m}$ and so $(1-x)\alpha(1) = \wt\mu(f_1)$. As $1-x$ is a unit in the local ring $R$, we have $\alpha(1) \in \im\ \wt\mu$. In other words, $\alpha(1) = \wt\mu(f)$ for some $f\in F$. Thus, $\mu_1(f,1) = \wt\mu(f) - \alpha(1) =0$. By exactness, $(f,1) = \nu_1(g)$ for some $g\in G$, and so $\wt\nu(g) = 1$, making $\wt\nu$ surjective.

Now suppose that $\wt\nu(g) = 1$ for some $g\in G$ and that $\im\ \nu \subseteq \mathfrak{m}F$. Then $0 = \mu_1\circ \nu_1 (g) = \mu_1(\nu(g),1) = \wt\mu(\nu(g)) - \alpha(1)$ and so $\alpha(1) \in \wt\mu(\im\ \nu) \subseteq \wt\mu( \mathfrak{m}F ) \subseteq \mathfrak{m}M$. 
\end{proof}

\begin{cor} \label{avoid}
Let $cl$ be an operation satisfying the big axioms. If $\alpha:R\to M$ is phantom, then $\alpha(1)\not\in \mathfrak{m}M$.
\end{cor}
\begin{proof}
Given $\alpha$, build a diagram (\ref{dgrm}) using a minimal free presentation so that $\wt\nu \in (\im\ \nu^\vee)^{cl}_{G^\vee}$, which we can do by the independence of free presentation implied by Lemma~\ref{indep}. By minimality, we may use the converse of the lemma above. For a contradiction, suppose that $\wt\nu(g^*) = 1$ for some $g^*\in G$.  Let $\phi: G^\vee\to R$ be the evaluation map $\phi(h) = h(g^*)$. Then $\phi(\wt\nu) = 1$. If $\eta \in F^\vee$, then $\phi(\eta\circ \nu) = \eta(\nu(g^*)) \in \mathfrak{m}$ as $\im\ \nu \subseteq \mathfrak{m}F$. Thus, $\phi(\im\ \nu^\vee) \subseteq \mathfrak{m}$. Axioms (3), (4), and (6) imply that $1 = \phi(\wt\nu) \in \mathfrak{m}^{cl}_R = \mathfrak{m}$, a contradiction. 
\end{proof}

\section{Phantom Extensions and Big Cohen-Macaulay Modules}

In this section, we show that phantom extensions can be modified into balanced big Cohen-Macaulay modules. This was also done in \cite[Section 3]{D10} but only for finitely-generated phantom extensions. We now extend the result to arbitrary phantom extensions and refine the argument to avoid the use of matrices and basis choices.
As in \cite{Ho75}, \cite{HH94}, and \cite{D10}, we will make use of module modifications in order to construct a balanced big Cohen-Macaulay module by forcing relations on systems of parameters to be trivialized in a sequence of modules. The key part is to show that one prevents the limit of the sequence $B$ from having $B = \mathfrak{m}B$. 
The first step for us though is to build a commutative diagram like (\ref{bigdgrm}) with some extra properties based on a relation on a system of parameters in the module $M$. 

\begin{disc} \label{sopconstr}
Let $M$ be an arbitrary $R$-module with $\alpha:R\to M$ an injection. Let $x_1\ldots, x_{k+1}$ be part of a system of parameters for $R$ and suppose that 
$$
x_1u_1+\cdots + x_{k+1}u_{k+1} = 0
$$ 
is a relation in $M$. Using the short exact sequence $0\to R\overset{\alpha}\to M\overset{\beta}\to Q\to 0$, we will build a free presentation of $Q$ as in (\ref{dgrm}), but we want to highlight an element that maps to $\overline{u_{k+1}} \in Q$. Let $F = F_1\oplus Rf^*$ be a direct sum of free modules that maps onto $Q$ via $\mu$ so that $\mu(f^*) = \overline{u_{k+1}}$. (We choose $F_1$ so that it already surjects onto $Q$ and think of $f^*$ as a redundant add-on.) Given this free $F$ and map $\mu$, extend it to a free presentation of $Q$ with a free module $G$.
We can then choose a lift $\wt\mu: F \to M$ of $\mu$ so that $\wt\mu(f^*) = u_{k+1}$ while $\wt\mu$ restricted to $F_1$ is lifted arbitrarily. Then construct the rest of the maps in (\ref{bigdgrm}) as in Section~2. 

Under these conditions, we study $\im\ \nu^\vee$ in $G^\vee$. Since $F = F_1\oplus Rf^*$, we have $F^\vee = (F_1)^\vee \oplus (Rf^*)^\vee$. Let $h^* \in (Rf^*)^\vee$ be the map $h^*(f^*) = 1$, and let $y = \nu^\vee(h^*) = h^*\circ \nu \in \im\ \nu^\vee$. Then $\im\ \nu^\vee = H+Ry$, where $H=\nu^\vee((F_1)^\vee)$. 

We have thus constructed a commutative diagram as in (\ref{bigdgrm}) with exact rows with the following extra conditions, which we list for ease of reference:
\begin{nlist} 
\item $F = F_1 \oplus Rf^*$ with $\mu(f^*) = \overline{u_{k+1}}$, $\wt\mu(f^*) = u_{k+1}$, and $\mu_1(f^*,0) = u_{k+1}$
\item $\im\ \nu^\vee = H+Ry = \nu^\vee((F_1)^\vee) + R(h^*\circ \nu)$, where $h^*(f^*) = 1$
\end{nlist}
\end{disc}

\begin{lemma} \label{sopmap}
Let $M$ be an arbitrary $R$-module with $\alpha:R\to M$ an injection. Let $x_1,\ldots, x_{k+1}$ be part of a system of parameters for $R$,  $I = (x_1,\ldots,x_k)$, and suppose that $x_1u_1+\cdots + x_{k+1}u_{k+1} = 0$ is a relation in $M$. Use the construction of Discussion~\ref{sopconstr} applied to diagram (\ref{bigdgrm}). Let $N = G^\vee/H$. Then there exists a 
map $\phi: N \to R/I$ with $\phi(y) = x_{k+1} + I$ and $\wt\nu \in \ker \phi$.
\end{lemma}
\begin{proof}
Since $x_1u_1+\cdots + x_{k+1}u_{k+1} = 0$ in $M$ and $\mu_1:F\oplus R\to M$ is surjective, there exists $(f_1,r_1),\ldots,(f_k,r_k)$ in $F\oplus R$ such that $\mu_1(f_i,r_i) = u_i$ (without loss of generality, we may assume that $f_i\in F_1\subseteq F$ since $F_1$ surjects onto $Q$). Thus, 
$$x_1(f_1,r_1)+\cdots + x_k (f_k,r_k) + x_{k+1}(f^*,0) \in \ker \mu_1 = \im\ \nu_1.$$ 
Let $g^*\in G$ be such that 
$$\nu_1(g^*) = x_1(f_1,r_1)+\cdots + x_k (f_k,r_k) + x_{k+1}(f^*,0),$$ 
and define a map $\psi: G^\vee\to R/I$ by $\psi(h) = h(g^*) + I$.  
We claim that $\psi(H) \subseteq I$, which would induce a 
map $\phi: N\to R/I$. Indeed, if $h\circ \nu \in H =  \nu^\vee((F_1)^\vee)$ with $h\in (F_1)^\vee$, then 
$$h\circ\nu (g^*) = h(x_1 f_1 + \cdots + x_{k} f_{k}+x_{k+1}f^*)\in I$$
 because $h(f^*) = 0$.

 Now we establish the claim about $\phi(y)$.  
 $$\psi(y) = y(g^*) = h^*\circ \nu(g^*) = h^*(x_1 f_1 + \cdots + x_{k} f_{k}+x_{k+1}f^*) = x_{k+1}+I$$
  as $h^*(f_i)=0$ since $f_i\in F_1$, and $h^*(f^*) = 1$ so that $\phi(y) = x_{k+1}+I$ too. 
  
Finally we show that $\wt\nu\in \ker \phi$. 
$$\psi(\wt\nu) = \wt\nu(g^*) + I = \alpha^{-1}\circ\wt\mu\circ\nu(g^*) + I = \alpha^{-1}\circ\wt\mu(x_1 f_1 + \cdots + x_{k} f_{k}+x_{k+1}f^*) + I,$$ 
but $\wt\mu(f_i) = \mu_1(f_i,r_i)+\alpha(r_i) = u_i + \alpha(r_i)$ and $\wt\mu(f^*) = u_{k+1}$. 
Therefore, 
$$\psi(\wt\nu) = \alpha^{-1}(x_1u_1+\cdots x_ku_k + x_{k+1}u_{k+1} + \alpha(x_1r_1+\cdots + x_kr_k)) + I = I$$ 
as  $x_1u_1+\cdots + x_{k+1}u_{k+1} = 0$. So, $\wt\nu\in \ker \phi$. 
\end{proof}

\begin{cor} \label{sopph}
Under the same assumptions and constructions of Lemma~\ref{sopmap}, assume also that $cl$ is a closure operation for $R$-modules that obeys the big axioms and that $\alpha:R \to M$ is phantom. Then $\wt\nu \in (H+Iy)^{cl}_{G^\vee}$. 
\end{cor}
\begin{proof}
As the extension is phantom, Lemma~\ref{indep} implies that $\wt\nu \in (\im\ \nu^\vee)^{cl}_{G^\vee} = (H+Ry)^{cl}_{G^\vee}$. By axiom (4), $\wt\nu +H \in (R\overline{y})^{cl}_N$. Then Lemma~\ref{sopmap} and Lemma~\ref{axiomprop}(c), the latter of which requires Axiom~(7), imply that $\wt\nu +H \in (I\overline{y})^{cl}_N$. Finally, by Lemma~\ref{axiomprop}(a), we can pull back to $G^\vee$ to get $\wt\nu \in (H+Iy)^{cl}_{G^\vee}$.
\end{proof}

Now we are ready to look at a module modification of $M$ which will trivialize the relation $x_1u_1+\cdots + x_{k+1}u_{k+1} = 0$. Starting with an arbitrary module $M$, one can produce a balanced big Cohen-Macaulay module through a very large direct limit of modules that are produced by forcing relations on systems of parameters to become trivial eventually. The difficult part is to show that one has not trivialized so much that the end module has $B = \mathfrak{m}B$. This problem can be avoided by keeping track of a special element that avoids being in $\mathfrak{m}M'$ for all modules $M'$ in the limit. In our case, we will start with an injection $\alpha:R\to M$ that is phantom and keep track of the image of $\alpha(1)$ throughout. We will keep this element out of $\mathfrak{m}M'$ by showing that all modifications remain phantom extensions of $R$ so that Corollary~\ref{avoid} gives the necessary exclusion. See \cite[Discussion 5.15]{HH94} for more background on the concept. 

\begin{disc} \label{modconstr}
Let $M$ be an arbitrary $R$-module with an injection $\alpha:R\to M$ and a relation $x_1u_1+\cdots + x_{k+1}u_{k+1} = 0$ in $M$ on a partial system of parameters $x_1,\ldots, x_{k+1}$ from $R$. Define the module modification of $M$ with respect to the relation given above by
$$
M' := \frac{M\oplus R^k}
{R(u_{k+1} , x_1,\cdots, x_k)}
$$
along with maps $f:M\to M'$ given by $u\mapsto \overline{(u,\mathbf{0})}$ and $\alpha': R\overset{\alpha}\to M\overset{f}\to M'$. Our goal is to prove that $\alpha'$ is a phantom extension when $\alpha$ is phantom. By the proof of \cite[Lemma 3.8]{D10}, $\alpha'$ is automatically injective.

Now, assume that $\alpha:R\to M$ is a phantom extension with a free presentation of $Q=M/\alpha(R)$ and commutative diagram as in (\ref{dgrm}). The big task will be to build a commutative diagram similar to (\ref{dgrm}) for $\alpha':R\to M'$. Let
$$
Q' = M'/\alpha'(R) = \frac{Q\oplus R^k}
{R(\overline{u_{k+1}} , x_1,\cdots, x_k)}
$$
Then $0\to R\overset{\alpha'}\to M'\overset{\beta'}\to Q'\to 0$, where $\beta'$ is the quotient map, gives a short exact sequence. Using the free presentation $G\overset{\nu}\to F\overset{\mu}\to Q\to 0$ of $Q$ from Discussion~\ref{sopconstr} with its special element $f^*\in F$,  we will build a free presentation of $Q'$ as $G\oplus R \overset{\nu'}\to F\oplus R^k \overset{\mu'}\to Q'\to 0$ via the maps 
$$
\mu'(f,r_1,\ldots,r_k) = \overline{(\mu(f),r_1,\ldots,r_k)}
$$ 
and 
$$
\nu'(g,r) = (\nu(g) + rf^*,rx_1,\ldots,r x_k) = \nu(g) + r(f^*,x_1,\ldots,x_d).
$$
We claim that this construction gives a free presentation for $Q'$. Indeed, since $\mu$ is surjective, $\mu'$ above will also surject onto $Q'$. To see exactness in the middle, note that 
$$
\begin{array}{rcl}
\mu'\circ\nu' (g,r) & = & \mu'(\nu(g) + rf^*,rx_1,\ldots,r x_k)  \\[2mm]
 & = & \overline{(\mu\circ\nu(g)+r\mu(f^*),rx_1,\ldots,r x_k)} \\[2mm]
& = & \overline{(r\overline{u_{k+1}},rx_1,\ldots,r x_k)} = 0.
\end{array}
$$ 
Additionally, if $\mu'(f,r_1,\ldots,r_k) = \overline{(\mu(f),r_1,\ldots,r_k)} = 0$ in $Q'$, then within $Q\oplus R^k$ we have 
$$(\mu(f),r_1,\ldots,r_k) = r(\overline{u_{k+1}} , x_1,\cdots, x_k)$$ 
for some $r\in R$. Hence, $r_i = rx_i$ for all $i$ within $R$, and $0 = \mu(f) - r\overline{u_{k+1}} = \mu(f-rf^*)$ within $Q$. These equations imply that there exists $g\in G$ such that $\nu(g) = f-rf^*$ and so 
$\nu'(g,r) = (\nu(g)+rf^*,rx_1,\ldots,rx_k) = (f,r_1,\ldots,r_k)$. 

We next join the short exact sequence $0\to R\overset{\alpha'}\to M'\overset{\beta'}\to Q'\to 0$ with the free presentation $G\oplus R \overset{\nu'}\to F\oplus R^k \overset{\mu'}\to Q'\to 0$ to form a commutative diagram. Lift $\mu'$ to a map $\wt\mu':F\oplus R^k\to M'$ such that $\wt\mu'$ restricted to $F$ is just the map $\wt\mu$, and $\wt\mu'$ restricted to $R^k$ is the same as $\mu'$ restricted to $R^k$. Then lift $\wt\mu'\circ \nu':G\oplus R\to \im\ \alpha' \subseteq M'$ to $\wt\nu':G\oplus R \to R$. We now have a commutative diagram with exact rows
$$
\xymatrix{ R \ar[r]^{\alpha'} & M' \ar[r]^{\beta'} & Q'\ar[r] & 0 \\
G\oplus R\ar[r]^{\nu'} \ar[u]_{\wt\nu'} & F\oplus R^k\ar[r]^{\mu'}\ar[u]_{\wt\mu'} & 
Q'\ar[r]\ar[u]_\id & 0} \leqno \deqno \label{moddgrm}
$$
\end{disc}

In order to prove that $\alpha':R\to M'$ is phantom when $\alpha: R\to M$ is phantom, we need to analyze  the condition $\wt\nu' \in (\im\ (\nu')^\vee)^{cl}_{(G\oplus R)^\vee}$.

\begin{lemma}
Use the notation and assumptions of Discussions~\ref{sopconstr}, Lemma!\ref{sopmap}, and \ref{modconstr}. We have the following:
\begin{alist}
\item $\wt\nu' = \wt\nu\oplus 0: G\oplus R\to R$
\item $\im\ (\nu')^\vee = (H\oplus 0) + R(y\oplus \id) + I(0\oplus \id)$ 
\item Suppose that $cl$ is a closure operation for $R$-modules that obeys the big axioms. If $\alpha:R\to M$ is phantom, then $\alpha':R\to M'$ is phantom.
\end{alist}
\end{lemma}
\begin{proof}
For (a), the construction above gives $\wt\nu' = (\alpha')^{-1}\circ \wt\mu' \circ \nu'$ as it does in (\ref{dgrm}). Note that 
$$
\begin{array}{rcl}
\wt\mu' \circ \nu'(g,r) & = & \overline{(\wt\mu\circ \nu(g) + r\wt\mu(f^*),rx_1,\ldots, rx_k)} \\[2mm]
& = & \overline{(\wt\mu\circ \nu(g),\mathbf{0})}+r\overline{(u_{k+1},x_1,\ldots, x_k)} = \overline{(\wt\mu\circ \nu(g),\mathbf{0})}.
\end{array}
$$ 
Thus, 
$$\wt\nu'(g,r) = (\alpha')^{-1}\overline{(\wt\mu\circ \nu(g),\mathbf{0})} = \alpha^{-1}\circ \wt\mu\circ \nu(g) = \wt\nu(g).$$

For (b), first recall from Discussion~\ref{sopconstr} that $F = F_1 \oplus Rf^*$, $H = \nu^\vee((F_1)^\vee)$, and $(\nu')^\vee: (F\oplus R^k)^\vee \to (G\oplus R)^\vee$. Let  $x_1,\ldots, x_{k+1}$ be part of a system of parameters for $R$, and let $I = (x_1,\ldots, x_k)$. Write $h\in (F\oplus R^k)^\vee = (F_1\oplus Rf^* \oplus R^k)^\vee$ as 
$$
h= h_1\oplus h_2 \oplus h_3 = h_1\oplus sh^* \oplus h_3,
$$ 
and write $\nu'(g,r) = \nu(g) + r(f^*,x_1,\ldots,x_k)$.  Then 
$$
\begin{array}{rcl}
(\nu')^\vee(h)(g,r) & = & h\circ \nu' (g,r) \\[2mm]
& = & h_1(\nu(g)) + sh^*(\nu(g))+rsh^*(f^*) + rh_3(x_1,\ldots,x_k) \\[2mm] 
& = & h_1(\nu(g)) + s(y(g)+r) + rh_3(x_1,\ldots,x_k).
\end{array}
$$ 
These three terms exactly generate $(H\oplus 0) + R(y\oplus \id) + I(0\oplus \id)$ given the form above and the fact that $I = (x_1,\ldots, x_k)$.

For (c), we assume that $\wt\nu \in (\im\ \nu^\vee)^{cl}_{G^\vee}$, and we need to show that $\wt\nu' \in (\im\ (\nu')^\vee)^{cl}_{(G\oplus R)^\vee}$. By Corollary~\ref{sopph} we know that $\wt\nu \in (H+Iy)^{cl}_{G^\vee}$. Thus, 
$$\wt\nu' = \wt\nu \oplus 0 \in (H+Iy)^{cl}_{G^\vee} \oplus 0 \subseteq (H+Iy)^{cl}_{G^\vee} \oplus 0^{cl}_{R^\vee}$$
by axiom (1). Then Lemma~\ref{axiomprop}(b) implies that $\wt\nu' \in ( (H+Iy)\oplus 0 )^{cl}_{(G\oplus R)^\vee}$. Next note that
$$(H+Iy)\oplus 0 \subseteq (H\oplus 0) + R(y\oplus \id) + I(0\oplus \id)= \im\ (\nu')^\vee$$
because 
$$
(h+iy)\oplus 0 = (h\oplus 0) + i(y\oplus 1) - i(0\oplus 1).
$$
Finally, axiom (3) achieves the result that  $\wt\nu' \in (\im\ (\nu')^\vee)^{cl}_{(G\oplus R)^\vee}$.
\end{proof}

We now have our main theorem of this section, which generalizes the main theorem of \cite{D10} to arbitrary $R$-modules $M$, not just finitely generated modules.

\begin{thm} \label{phmodbigCM}
Let $R$ be a local domain possessing an operation $cl$ that obeys the big axioms. If $\alpha:R\to M$ is a phantom extension, then $M$ can be modified into a balanced big Cohen-Macaulay module over $R$. 
\end{thm}

As an add-on fact, we point out the following lemma regarding modules that factor through phantom extensions.

\begin{lemma}
Let $R$ be a local domain with a closure operation $cl$ obeying the big axioms, and let $M$ and $M_1$ be arbitrary $R$-modules. Suppose that $R \overset{\alpha_1}\to M_1 \overset{\theta}\to M$ is a sequence of maps such that the composition $\alpha = \theta\circ \alpha_1$ is phantom via $cl$. Then $\alpha_1$ is also phantom via $cl$. 
\end{lemma}
\begin{proof}
Consider the diagram of maps shown below. One may want to visualize this diagram by rolling it vertically and gluing the identical top and bottom rows together to produce a commutative prism. Below, we will justify that all squares are commutative, and the rows are exact. We also show that the vertical paths are commutative in the sense that composing a pair of vertical maps from the top row to the middle row equals the composition of vertical maps from the bottom row to the middle row within the same column. 
$$
\xymatrix{
& G_1\ar[r]^{\nu_1} \ar[d]^{\wt\nu_1} & F_1\ar[r]^{\mu_1}\ar[d]^{\wt\mu_1} & Q_1\ar[r]\ar[d]^\id & 0\\ 
0\ar[r] & R \ar[r]^{\alpha_1}\ar[d]^\id & M_1 \ar[r]^{\beta_1}\ar[d]^\theta & Q_1\ar[r]\ar[d]^{\overline{\theta}} & 0 \\
0\ar[r] & R \ar[r]^\alpha & M \ar[r]^\beta & Q\ar[r] & 0 \\
& G\ar[r]^\nu \ar[u]_{\wt\nu} & F\ar[r]^\mu\ar[u]_{\wt\mu} & Q\ar[r]\ar[u]_\id & 0\\ 
& G_1\ar[r]^{\nu_1} \ar[u]_{\psi} & F_1\ar[r]^{\mu_1}\ar[u]_{\phi} & Q_1\ar[r]\ar[u]_{\overline{\theta}} & 0
}  \leqno \deqno \label{phfactor}
$$

We begin by explaining the construction of the diagram. Start with the maps $\alpha_1$ and $\alpha$. As $\alpha$ is injective (because $\alpha$ is phantom), $\alpha_1$ is also injective. Use $\theta$ to build SESs and commutative squares as in rows 2 and 3 above. Next, choose a free presentation of $Q$ and lifting maps as in row 4, which shows that the top four rows of maps have been defined and that the top three layers of squares are commutative. 

Next, we will build a free presentation for $Q_1$, but we need to be careful so that it is compatible with the presentation of $Q$. Start with an arbitrary choice of free presentation as in rows 1 and 5 (same modules and maps). We claim that we can lift $\mu_1$ to $\wt\mu_1:F_1\to M_1$ so that $\theta\circ \wt\mu_1:F_1 \to \wt\mu(F)\subseteq M$. Indeed, let $f_1$ be a free basis element for $F_1$ over $R$. Then $\mu_1(f_1) = \beta_1(m_1)$ for some $m_1\in M_1$, and $\overline{\theta}\circ\mu_1(f_1) = \mu(f)$ for some $f\in F$. Therefore, 
$$
\beta(\wt\mu(f) - \theta(m_1))  =  \mu(f) - \overline{\theta}\circ \beta_1(m_1) 
 =  \overline{\theta}\circ\mu_1(f_1) -  \overline{\theta}\circ \mu_1(f_1) = 0
$$
and so $\wt\mu(f) - \theta(m_1) = \alpha(r)$ for some $r\in R$. Define $\wt\mu_1(f_1) = m_1 +\alpha_1(r)$. Then 
$$
\theta\circ \wt\mu_1(f_1) = \theta(m_1) + \theta\circ\alpha_1(r) = [\wt\mu(f) - \alpha(r)] + \alpha(r) \in \wt\mu(F) \subseteq M
$$
Given this definition of $\wt\mu_1$, we lift $\nu_1$ to $\wt\nu_1:G_1\to R$ without extra conditions.

To complete the bottom layer of commutative squares, we define the maps $\phi$ and $\psi$. To define $\phi:F_1\to F$ as a lift of $\mu_1$ we note that $\theta\circ \wt\mu_1: F_1 \to \wt\mu(F)$ and $\wt\mu:F\to M$ gives a surjection onto $\wt\mu(F)$. Hence, we can lift $\theta\circ \wt\mu_1$ to a map $\phi:F_1\to F$ such that $\theta\circ \wt\mu_1 = \wt\mu\circ \phi$. We can then lift $\nu_1$ to $\psi:G_1\to G$ without extra conditions. 

At this point all of the maps in the diagram have been defined and all of the squares have been shown to be commutative.

We now show that the vertical paths from top to middle and bottom to middle commutate with each other in the sense described at the start of the proof. We have already noted that $\theta\circ \wt\mu_1 = \wt\mu\circ \phi$ by construction. Also, 
$$
\alpha\circ \wt\nu\circ\psi = \wt\mu\circ \phi\circ \nu_1 = \theta\circ \wt\mu_1\circ \nu_1 = \alpha\circ \wt\nu_1.
$$
As $\alpha$ is injective, we have $\wt\nu\circ\psi =\wt\nu_1$.

As a result, we can now prove that $\alpha$ phantom implies that $\alpha_1$ is phantom. As $\alpha$ is phantom, $\wt\nu \in (\im\ \nu^\vee)^{cl}_{G^\vee}$. Applying $\psi^\vee: G^\vee\to G_1^\vee$, we have 
$$
\wt\nu_1 = \psi^\vee(\wt\nu) \in (\psi^\vee(\im\ \nu^\vee))^{cl}_{G_1^\vee} = (\im(\nu\circ\psi)^\vee)^{cl}_{G_1^\vee} = (\im(\phi\circ\nu_1)^\vee)^{cl}_{G_1^\vee} \subseteq (\im\ \nu_1^\vee)^{cl}_{G_1^\vee}
$$
by using axioms (3) and (4), which shows that $\alpha_1$ is phantom. 
\end{proof}

\section{Solidity and Phantom Extensions}

Hochster \cite{Ho94} defined an $R$-module $M$ to be solid if $\Hom_R(M,R)\ne 0$. If the module is an $R$-algebra $S$, then $S$ is solid if and only if there exists an $R$-linear map $\gamma:S\to R$ such that $\gamma(1) = c\ne 0$. While Hochster showed that all balanced big Cohen-Macaulay modules and algebras are solid, an open question from \cite{Ho94} was whether or not all solid algebras can be modified into balanced big Cohen-Macaulay algebras, i.e., are solid algebras also seeds ($R$-algebras that can be mapped to a balanced big Cohen-Macaulay $R$-algebra \cite{D07}). In \cite[Example 3.14]{HH09} Hochster and Huneke gave a counterexample in characteristic $p$, which shows that not all solid algebras in characteristic $p$ are seeds even though all seeds are solid. 

In this section we provide a complementary result to the connection between solid algebras and seeds in light of the counterexample mentioned above. We will prove that if $R$ is characteristic $p$ and also $F$-finite (for $R$ reduced, $R\to R^{1/p}$ is a module-finite extension), then all solid algebras are phantom extensions of $R$ and thus can be modified into big Cohen-Macaulay modules. We end the section with an example of a solid module that is not a phantom extension of $R$.

We will need several characteristic $p$ lemmas to help set up everything.

\begin{lemma}(cf. \cite[Proposition 5.12]{HH94}) \label{pses}
Let $R$ be reduced and have characteristic $p$ with a short exact sequence $0\to R\overset{\alpha}\to M\overset{\beta}\to Q\to 0$. Then for any $q=p^e$, 
$0\to R^{1/q}\overset{1\otimes \alpha}\to R^{1/q}\otimes_R M\overset{1\otimes \beta}\to R^{1/q}\otimes_R Q\to 0$ is still a short exact sequence. 
\end{lemma}

\begin{lemma} \label{psolidmap}
Let $R$ be a domain of characteristic $p$. If $S$ is a solid $R$-algebra, then there exists a fixed $c\neq 0$ in $R$ such that for all $q=p^e$, there is an $R$-linear map $\gamma_q: R^{1/q}\otimes_R S \to R^{1/q}$ with $\gamma_q(r^{1/q}\otimes 1) = c^{1/q}r^{1/q}$.
\end{lemma}
\begin{proof}
Since $S$ is solid, there is an $R$-linear map $\gamma:S\to R$ such that $\gamma(1) = c \neq 0$. Using this fixed $c$ (independent of $q$), we will build the other maps. For each $q$, consider the $R$-bilinear map $R^{1/q}\times S \to R^{1/q}$ given by $(r^{1/q},s)\mapsto r^{1/q}\gamma(s^q)^{1/q}$. This map induces the required $\gamma_q$.
\end{proof}

\begin{lemma}
Let $R$ be a Noetherian ring (of any characteristic) with an $R$-algebra $\rho:R\to T$ such that $T$ is module-finite over $R$. Let $G$ be a free (but not necessarily finite) $R$-module. Then $T\otimes_R \Hom_R(G,R) \isom \Hom_T(T\otimes_R G, T)$. 
\end{lemma}
\begin{proof}
Define 
$$\phi: T\otimes_R \Hom_R(G,R) \to \Hom_T(T\otimes_R G, T) \mbox{ by } \phi(t\otimes h) = \wt{h},$$
where $\wt{h}(t'\otimes g) = tt'\rho(h(g))$ for all $h\in \Hom_R(G,R)$. To see that $\phi$ is one-to-one, suppose that $\phi(\sum_i t_i\otimes h_i) = 0$-map. Then for all $g \in G$, we have 
$$0 = \sum_i \wt{h_i}(1\otimes g) = \sum_i t_i\rho(h_i(g)),$$
 which implies that  
$$\sum_i  t_i\otimes h_i(g) = \sum_i t_i\rho(h_i(g))\otimes 1 = 0.$$
 Thus, $\sum_i t_i\otimes h_i$ is also the 0-map in $T\otimes_R \Hom_R(G,R)$. To see that $\phi$ is onto, let $\{\epsilon_i\}_{i\in \mathcal{I}}$ be a basis for $G$ over $R$ so that $\{1\otimes \epsilon_i\}_{i\in \mathcal{I}}$ is a basis for $T\otimes_R G$ over $T$. Let $\wt{h} \in \Hom_T(T\otimes_R G, T)$ be such that $\wt{h}(1\otimes \epsilon_i) = t_i \in T$. Let $J$ be the $T$-ideal generated by the $t_i$. As $T$ is module-finite over $R$ and $R$ is Noetherian, $J$ is a finite $R$-module with generators $\{b_1,\ldots b_n\} \subseteq T$. For all $i\in \mathcal{I}$, write $t_i = \sum_j \rho(r_{ij})b_j$. For $1\leq j\leq n$ define maps $h_j:G\to R$ by $h_j(\epsilon_i) = r_{ij}$. Thus, $\sum_{j=1}^{n} b_j\otimes h_j$ is in $T\otimes_R \Hom_R(G,R)$, and 
 $$\sum_j b_j\otimes h_j(\epsilon_i) = \sum_j b_j\otimes r_{ij} = \sum_j b_j\rho(r_{ij})\otimes 1 = t_i\otimes 1.$$ 
 Thus, $\wt{h} = \phi(\sum_j b_j\otimes h_j)$. 
\end{proof}

\begin{cor} \label{Ffinisom}
Let $R$ be a reduced, $F$-finite Noetherian ring of characteristic $p$, and let $G$ be a free $R$-module. Then $R^{1/q}\otimes_R \Hom_R(G,R) \isom \Hom_{R^{1/q}}(R^{1/q}\otimes_R G, R^{1/q})$.
\end{cor}

\begin{disc}
Let $(R,\mathfrak{m})$ be a local domain, and let $S$ be a solid $R$-algebra via the map $\alpha:R\to S$. Then $R/I$ is not solid for any proper ideal $I$ and so $\alpha$ must be an injection. Thus, we have a short exact sequence $0\to R\overset{\alpha}\to S\overset{\beta}\to Q\to 0$ with an $R$-linear map $\gamma: S\to R$ such that $\gamma(1) = c\ne 0$. 

If $R$ is also characteristic $p$, then Lemmas~\ref{pses} and \ref{psolidmap} imply that for each $q=p^e$, we have a short exact sequence $0\to R^{1/q}\overset{1\otimes \alpha}\to R^{1/q}\otimes_R S\overset{1\otimes \beta}\to R^{1/q}\otimes_R Q\to 0$ with an $R$-linear map $\gamma_q: R^{1/q}\otimes_R S\to R^{1/q}$ such that $\gamma_q(r^{1/q}\otimes 1) = c^{1/q}r^{1/q}$, i.e., $\gamma_q\circ (1\otimes \alpha)$ is multiplication by $c^{1/q}$. So, for each $q$ we can build a diagram
$$
\xymatrix{ R^{1/q} \ar[r]^{1\otimes \alpha} & R^{1/q}\otimes_R S \ar[r]^{1\otimes \beta} & R^{1/q}\otimes_R Q\ar[r] & 0 \\
R^{1/q}\otimes_R G\ar[r]^{1\otimes \nu} \ar[u]_{1\otimes \wt\nu} & R^{1/q}\otimes_R F\ar[r]^{1\otimes \mu}\ar[u]_{1\otimes \wt\mu} & 
R^{1/q}\otimes_R Q\ar[r]\ar[u]_\id & 0} \leqno \deqno \label{qdgrm}
$$
which continues to be commutative with exact rows so that the bottom row is a free presentation of $R^{1/q}\otimes_R Q$ over $R^{1/q}$. Thus, 
$$
c^{1/q}(1\otimes \wt\nu) = \gamma_q\circ (1\otimes \alpha)\circ (1\otimes \wt\nu) = \gamma_q \circ (1\otimes \wt\mu)\circ (1\otimes \nu) = (1\otimes \nu)^\vee(\gamma_q \circ (1\otimes \wt\mu))
\leqno \deqno \label{compos}
$$
and so $c^{1/q}\otimes \wt\nu \in \im((1\otimes \nu)^\vee)$ for all $q$. In other words, 
$$
c^{1/q}\otimes \wt\nu \in \im (\Hom_{R^{1/q}}(R^{1/q}\otimes_R F, R^{1/q})\to \Hom_{R^{1/q}}(R^{1/q}\otimes_R G, R^{1/q}))
$$
but Corollary~\ref{Ffinisom} means that for all $q$
$$
c^{1/q}\otimes \wt\nu \in \im (R^{1/q}\otimes_R \Hom_R(F,R)\to R^{1/q}\otimes_R \Hom_R(G,R))
$$
Hence, $\wt\nu \in (\im\ \nu^\vee)^*_{G^\vee}$, the tight closure of $\im\ \nu^\vee$ inside $G^\vee$, using tight closure defined for all $R$-modules, which satisfies the big axioms by Proposition~\ref{tcbigaxioms}. Thus the solid algebra $S$ is a phantom extension of $R$. We therefore arrive at the following theorem.
\end{disc}

\begin{thm} \label{solidalgph}
Let $(R,\mathfrak{m})$ be an $F$-finite, local domain of characteristic $p$. If $S$ is a solid $R$-algebra, then $S$ is a phantom extension of $R$ via tight closure defined for all $R$-modules as in Proposition~\ref{tcbigaxioms}.
\end{thm}

\begin{cor} \label{solidalgbigCM}
Let $(R,\mathfrak{m})$ be an $F$-finite, local domain of characteristic $p$ possessing a big test element (e.g., $R$ is also complete). If $S$ is a solid $R$-algebra, then $S$ can be modified into a balanced big Cohen-Macaulay module over $R$. 
\end{cor}
\begin{proof}
Apply Theorem~\ref{phmodbigCM} to the fact that tight closure in characteristic $p$ obeys the big axioms in the presence of big test elements.
\end{proof}

The corollary above stands in contrast to Hochster and Huneke's example in \cite[Example 3.14]{HH09} of a solid algebra in characteristic $p$ that cannot be modified into a big Cohen-Macaulay algebra. In the $F$-finite case, we can now see that although there are solid algebras that are not seeds (cannot be mapped to big Cohen-Macaulay algebras), all solid algebras can be modified into big Cohen-Macaulay modules because solid algebras are phantom extensions. We next point out that if $\alpha:R\to M$ is a phantom extension,
then $M$ is solid. A discussion\footnote{These ideas arose during a conversation with Rebecca R.G.} is needed first that connects the definition of a phantom extension to the long exact sequence for Ext.

\begin{disc} \label{LESExt}
Let $(R,\mathfrak{m})$ be a local domain. Consider the setup in Discussion~\ref{freepres} but with the free presentation of $Q$ continued one step further:
$$
\xymatrix{ 0\ar[r] & R \ar[r]^\alpha & M \ar[r]^\beta & Q\ar[r] & 0 \\
H\ar[u] \ar[r]^{\psi} & G\ar[r]^\nu \ar[u]_{\wt\nu} & F\ar[r]^\mu\ar[u]_{\wt\mu} & 
Q\ar[r]\ar[u]_\id & 0} 
$$
which we alter to the diagram below with short exact sequences with $\overline{G} = G/\im\ \psi$ (and abusing notation slightly by using $\nu$ and $\wt\nu$ for the maps in both cases):
$$
\xymatrix{ 0\ar[r] & R \ar[r]^\alpha & M \ar[r]^\beta & Q\ar[r] & 0 \\
0 \ar[r] & \overline{G}\ar[r]^\nu \ar[u]_{\wt\nu} & F\ar[r]^\mu\ar[u]_{\wt\mu} & 
Q\ar[r]\ar[u]_\id & 0}
$$

After applying $\Hom(-,R) = (-)^\vee$ to the diagram above, we have long exact sequences in Ext-modules and commutative squares:
$$
\xymatrix{ 
0\ar[r] & Q^\vee \ar[d]\ar[r] & M^\vee \ar[d]\ar[r] & R^\vee \ar[d]^{\wt\nu^\vee}\ar[r]^{\Delta\hspace{5mm}} & \textnormal{Ext}^1_R(Q,R) \ar[d]^\id\ar[r] &  \textnormal{Ext}^1_R(M,R)\ar[d]\ar[r] & 0\\
0 \ar[r] & Q^\vee \ar[r] & F^\vee \ar[r]^{\nu^\vee} & \overline{G}^\vee \ar[r]^{\delta\hspace{5mm}} &  \textnormal{Ext}^1_R(Q,R) \ar[r] & 0
} 
$$
where $\delta$ and $\Delta$ are the connecting homomorphisms, we identify $\textnormal{Ext}^1_R(Q,R) = \overline{G}^\vee/\im(\nu^\vee)$, and the modules on the far rights are zero due to the fact that $R$ and $F$ are free modules. With our identification of  $\textnormal{Ext}^1_R(Q,R)$, the connecting map $\delta$ is just a projection of $\overline{G}^\vee$ onto a quotient and so $\Delta$ is a composition of $\wt\nu^\vee$ with that projection. In other words, the connecting homomorphism $\Delta: R^\vee \to  \textnormal{Ext}^1_R(Q,R)$ is induced naturally by $\wt\nu^\vee$, and so by a slight abuse of notation we identify $\Delta$ with $\wt\nu^\vee$.
See \cite[III.1, III.2, XIV.1]{CE} for further details on the diagram chasing and the identification of extensions of modules with $\textnormal{Ext}^1$ based on projective resolutions to justify everything here.
\end{disc}

\begin{prop} \label{phsolid}
Let $(R,\mathfrak{m})$ be a local domain. If $\alpha:R\to M$ is a phantom extension, then $M$ is a solid $R$-module.
\end{prop}
\begin{proof}
We prove the contrapositive. Suppose that $M$ is not solid so that $M^\vee = \Hom(M,R) = 0$. Then the previous discussion gives the short exact sequence
$$
0 \to R \overset{\wt\nu^\vee}\to \textnormal{Ext}^1_R(Q,R) \to \textnormal{Ext}^1_R(M,R) \to 0
$$
where we identify $R^\vee$ with $R$ and the connecting map $R \to  \textnormal{Ext}^1_R(Q,R)$ with $\wt\nu^\vee$. The map $\wt\nu^\vee$ is injective if and only if $\wt\nu^\vee(r)\neq 0$ for all $r\neq 0$ if and only if $r\wt\nu^\vee(1) \neq 0$ if and only if $r\wt\nu \neq 0$ within $\Ext^1(Q,R) = \overline{G}^\vee/\im(\nu^\vee)$ if and only if $r\wt\nu \not\in \im\ \nu^\vee$ for any $r\neq 0$. By Lemma~\ref{axiomprop}(f) we have $\wt\nu \not\in (\im\ \nu^\vee)^{cl}$ calculated within any ambient module. In particular, $\wt\nu \not\in (\im\ \nu^\vee)^{cl}_{G^\vee}$ and so $\alpha$ is not phantom.
\end{proof}

While phantom extensions yield solid modules, not all solid modules are phantom extensions as the example below shows. We state the example below in general for any closure operation $cl$ that satisfies the big axioms, but one may think about tight closure as a specific example. The generality of the example comes from \cite[Theorem 5.9]{RG15a} which states that in regular rings all closure operations that satisfy the axioms are trivial.

\begin{exam} \label{solidnotph}
Let $R$ be a regular local domain of characteristic $p$ and dimension at least 2. Let $I$ be any ideal minimally generated by at least 2 elements, and let $\alpha:R\to I$ be an injection with $1$ mapping to one of the minimal generators of $I$. The inclusion $I\subseteq R$ shows that $I$ is a solid module over $R$. Moreover, we claim that $I$ is not a phantom extension via $cl$.  Indeed, $\alpha$ is not a split map as $I = \alpha(R)\oplus W$ implies that $W$ must be a torsion module since $I\subseteq R$ is a rank 1 torsion-free submodule. However, this would imply that $W = 0$, i.e., $I = \alpha(R)$, which is impossible since $I$ is minimally generated by at least 2 elements. Since $R$ is regular, \cite[Theorem 5.9]{RG15a} implies that the only closure operation $cl$ that satisfies the finite axioms will be the trivial closure. If $\alpha:I\to R$ were a phantom extension with respect to $cl$, then the corresponding diagram (\ref{dgrm}) would show that $\wt\nu \in (\im\ \nu^\vee)^{cl}_G =  \im\ \nu^\vee$. We show that this condition implies that $\alpha$ would be split, contradicting the work above. Indeed, $\wt\nu = h\circ \nu$ for some $h:F\to R$ implies that we can define $\theta:Q\to I$ by $\theta(q) = \wt\mu(f) - \alpha\circ h(f)$, where $Q = R/\alpha(I)$ and $f\in F$ such that $\mu(f) = q$. (The commutativity of (\ref{dgrm}) shows that $\theta$ is well-defined and independent of the choice of $f$: Suppose that $\mu(f') = q$ too. Then $f-f' = \nu(g)$ for some $g\in G$, and 
$$\wt\mu(\nu(g))) - \alpha\circ h(\nu(g)) = \alpha\circ \wt\nu (g) - \alpha \circ \wt\nu(g) =0$$
shows that $\wt\mu(f) - \alpha\circ h(f) = \wt\mu(f') - \alpha\circ h(f')$.) As 
$$\beta\circ \theta (q) = \beta\circ \wt\mu(f) - \beta\circ\alpha\circ h(f) = \mu(f) - 0 = q$$
 for all $q\in Q$, we see that $\beta$ (and so $\alpha$) would be split maps. Therefore, $I$ is a solid $R$-module that is not a phantom extension of $R$.
\end{exam}

For a complete, local domain $R$, we have the following implications, where the middle arrow is only known for characteristic $p$, $F$-finite rings:
$$
\textrm{seed algebra} \implies \textrm{solid algebra} \dashrightarrow \textrm{phantom extension} \implies \textrm{solid module}
$$
with examples that the reverse directions are false. (For the middle arrow, one can choose a module modification of $R$ itself as it is a phantom extension of $R$ but has lost the algebra structure in the process.)

\section{Open Questions}

We close with some open questions that are of great interest for later study. 

\begin{nlist}
\item Suppose that $\alpha:R\to B$ is a map to a big Cohen-Macaulay module $B$. For which maps, if any, is $\alpha$ a phantom extension? 
\item Are solid algebra maps $\alpha: R\to S$ always phantom extensions? In other words, can we remove the $F$-finite (or characteristic $p$) hypotheses of Theorem~\ref{solidalgph} using other characteristic $p$ methods or by only employing the characteristic-free axioms?
\item Can either of the above questions be answered by adding new axioms to the list presented in this article? Any new axiom should follow from the existence of big Cohen-Macaulay modules and preferably also be a property held by tight closure.
\item Can one define closure operation axioms that lead to big Cohen-Macaulay algebras? In a recent paper \cite{RG15b}, Rebecca R.G.\ has produced a positive answer to the this question when one adds an additional axiom involving preservation of a phantom extension after passing to a second symmetric power of a module.
\end{nlist}

We close by acknowledging the amazing, recent work related to big Cohen-Macaulay modules and algebras accomplished via perfectoids, including\cite{A16}, \cite{B16}, \cite{HM17}. These works have proven that big Cohen-Macaulay algebras exist in general mixed characteristic cases, and we are eager to see how increased understanding of perfectoids may shed light on families of closure operations in mixed characteristic as it relates to these axioms and phantom extensions. 

\section*{Acknowledgments}
I thank Mel Hochster who suggested the proof for Proposition~\ref{finext}(2). Also, I thank both Mel and Rebecca R.G.\ who suggested the proof to Lemma~\ref{axiomprop}(e). A conversation with Rebecca R.G. on a separate problem also led to Lemma~\ref{axiomprop}(f), Discussion~\ref{LESExt}, and significant improvements to an earlier version of Proposition~\ref{phsolid}. I also thank the anonymous referees for all helpful comments, especially the suggestion to rework Example~\ref{solidnotph} for any closure operation satisfying the axioms (not just tight closure). 

%


\end{document}